\documentclass[12pt,a4paper]{article}
\usepackage[a4paper,margin=2.6cm]{geometry}
\usepackage{amsmath,amssymb,amsthm,mathtools}
\usepackage{enumitem}
\usepackage{xcolor}
\usepackage[unicode,colorlinks=true,linkcolor=blue!60!black,citecolor=blue!60!black,urlcolor=blue!60!black]{hyperref}
\hypersetup{pdftitle={Compactness of Toeplitz Operators on the Bergman Space}}
\newtheorem{theorem}{Theorem}[section]
\newtheorem{lemma}[theorem]{Lemma}
\newtheorem*{conjecture}{Conjecture}
\theoremstyle{remark}

\newcommand{\D}{\mathbb D}
\newcommand{\dd}{\,\mathrm d}
\newcommand{\dA}{\,\mathrm dA}
\newcommand{\A}{A^2(\D)}
\newcommand{\norm}[1]{\left\lVert#1\right\rVert}
\newcommand{\abs}[1]{\left\lvert#1\right\rvert}
\DeclareMathOperator{\id}{id}
\DeclareMathOperator{\Span}{span}
\title{Compactness of Toeplitz Operators on the Bergman Space}
\author{%
Guangfu Cao, \hskip5mm Li He, \hskip5mm Shuqing Zhang$^1$\\
School of Mathematics and Information Science,\\
Guangzhou University,\\
Guangzhou 510006, China\\
\\[1ex]}
\date{}
\begin{document}
\maketitle
\footnotetext[1]{Corresponding authors.}

\begin{abstract}
Let $\varphi\in L^\infty(\D)$. We study compactness criteria for \(T_\varphi\) on the Bergman space $A^2(\D)$. Axler and Zheng~\cite{AZ1998} established a necessary and sufficient condition for compactness in terms of the Berezin transform. However, for a general bounded measurable function $\varphi $, 
its Berezin transform $\tilde{\varphi}$ does not readily reflect the local behavior  of $\varphi$. Motivated by a characterization in terms of the symbol itself, Zhu~\cite{ZhuSlides} proposed a conjecture on compact Toeplitz operators. In this paper, we characterize compactness of $T_\varphi$ on the unweighted Bergman space in terms of local averages of the symbol. We prove that compactness is equivalent to the vanishing of averages over Bergman disks of any prescribed fixed radius. We also establish an equivalent criterion in terms of Carleson box averages that tend to zero uniformly in the angular variable. Finally, we construct a nonnegative bounded symbol whose Carleson box averages tend to zero at every fixed angle, although the associated Toeplitz operator is not compact. 
\end{abstract}
\section{Introduction}
Let $\D=\{z\in\mathbb C:|z|<1\}$, and let
\[
\mathrm dA(z)=\frac1\pi\dd x\dd y,\qquad z=x+iy,
\]
denote the normalized area measure on the unit disk $\D$.
For a measurable set $E\subset\D$, write $|E|=\int_E\dA$. The unweighted Bergman space is
\[
\A=\left\{f\in H(\D):\norm f_2^2=\int_\D|f(z)|^2\dA(z)<\infty\right\},
\]
equipped with the inner product
\[
\langle f,g\rangle=\int_\D f(z)\overline{g(z)}\dA(z).
\]
Let $P:L^2(\D,\mathrm dA)\to\A$ be the orthogonal projection. For $\varphi\in L^\infty(\D)$, define
\[
T_\varphi f=P(\varphi f),\qquad f\in\A.
\]
This operator is bounded, with $\norm{T_\varphi}\le\norm\varphi_\infty$. These definitions and basic properties can be found in~\cite{ACM1982,HKZ2000,Zhu2007}.

For $a\in\D$, write
\[
\sigma_a(w)=\frac{a-w}{1-\overline a w},
\qquad
E_\rho(a)=\{z\in\D:|\sigma_a(z)|<\rho\},\quad 0<\rho<1.
\]
To fix the normalization of the metric, we set
\[
\beta(a,z)=\operatorname{arctanh}|\sigma_a(z)|,
\qquad D(a,r)=\{z\in\D:\beta(a,z)<r\}.
\]
Thus $D(a,r)=E_{\tanh r}(a)$. Multiplying the Bergman distance by a fixed positive constant does not affect the statements below concerning all fixed radii. Let $S_h$ denote a Carleson box of scale $h$. For $0<h<1$, set
\[
S(\theta,h)=\{se^{it}:1-h<s<1,\ |t-\theta|<h\},
\]
where angles are understood modulo $2\pi$. The radial height is $h$, and the full angular width is $2h$.

The general theory of Toeplitz operators on Bergman spaces goes back, in particular, to Axler, Conway, and McDonald~\cite{ACM1982}. 
Compactness of Bergman space Toeplitz operators has been studied through positivity, automorphism localization, and the Berezin transform. Positive symbols and measures admit criteria in terms of Berezin transforms and local averages; see Zhu~\cite{Zhu1988} and the accounts in~\cite{HKZ2000,Zhu2007}. Stroethoff and Zheng~\cite{SZ1992} obtained compactness criteria for Toeplitz and Hankel operators with bounded symbols on the ball and polydisk, including the unit disk. Korenblum and Zhu~\cite{KZ1995} used Tauberian theorems to characterize compactness for bounded radial symbols. Further compactness results are given by Stroethoff~\cite{Stroethoff1998}; see also Miao and Zheng~\cite{MiaoZheng2004} for compact operators on Bergman spaces and related operator-theoretic characterizations. Compactness of Toeplitz  operators has also been approached from two other directions: the operator-algebraic viewpoint of Salinas, Sheu, and Upmeier~\cite{SSU1989} (foliation $C^*$-algebras on pseudoconvex domains), and the function-theoretic viewpoint of Axler~\cite{Axler1986} (Bloch space and commutators of multiplication operators, including the compactness of the associated Hankel operators).

For general bounded symbols on the disk, Axler and Zheng~\cite{AZ1998} proved that a finite sum of finite products of Toeplitz operators is compact if and only if its Berezin transform vanishes at the boundary. In particular, if
\[
k_a(z)=\frac{1-|a|^2}{(1-\overline a z)^2},
\qquad
\widetilde\varphi(a)=\langle T_\varphi k_a,k_a\rangle,
\]
then
\[
T_\varphi\text{ is compact}
\quad\Longleftrightarrow\quad
\widetilde\varphi(a)\longrightarrow0\quad(|a|\to1).
\]
Su\'arez~\cite{Suarez2007} subsequently established essential-norm estimates and a Berezin-transform compactness criterion for operators in the norm-closed Toeplitz algebra on $A^p(\mathbb B_n)$, $1<p<\infty$; the disk Hilbert-space case is included.

Localization methods provide another perspective on this circle of results. Isralowitz, Mitkovski, and Wick~\cite{IMW2015} obtained compactness criteria on Bergman and Fock spaces using the Berezin transform and reproducing kernels under suitable localization hypotheses. Xia~\cite{Xia2015} further investigated localization and the Toeplitz algebra on the Bergman space.

There are also important extensions beyond bounded symbols. Zorboska~\cite{Zorboska2003} characterized boundedness and compactness for symbols in $\mathrm{BMO}^1$ by means of the Berezin transform. Cima and \v{C}u\v{c}kovi\'c~\cite{CC2005} constructed compact Toeplitz operators with unbounded symbols. Taskinen and Virtanen~\cite{TV2018} obtained compactness criteria on $A^p(\D)$ for integrable symbols satisfying an additional weak averaging condition.  For positive measure symbols on weighted Bergman spaces of the disk, Pel\'aez and R\"atty\"a~\cite{PR2016} characterized Schatten-class Toeplitz operators for radial weights satisfying a tail-doubling condition. 

The purpose of the present paper is not to replace the known Berezin-transform characterization, but to obtain equivalent criteria directly in terms of these local averages of a bounded, possibly complex-valued symbol. Cancellation must be retained: the averages involve $\varphi$. In the talk \emph{Ten Open Problems}, Zhu proposed the following conjecture, formulated in terms of the symbol itself, concerning compact Toeplitz operators on the unweighted Bergman space of the unit disk; see~\cite{ZhuSlides}.
\begin{conjecture}
Let $\varphi\in L^\infty(\D)$, and fix an arbitrary $r_0>0$.
Then the following three conditions are equivalent:
\begin{enumerate}[label=\textnormal{(C\arabic*)}]
\item $T_\varphi$ is compact on $\A$;
\item for the \emph{arbitrarily prescribed radius} $r_0$,
\begin{equation}\label{eq:conj-disk}
\lim_{|a|\to1}\frac1{|D(a,r_0)|}
\int_{D(a,r_0)}\varphi(z)\dA(z)=0;
\end{equation}
\item the Carleson box averages satisfy
\begin{equation}\label{eq:conj-box}
\lim_{h\downarrow0}
\frac1{|S(\theta,h)|}\int_{S(\theta,h)}\varphi(z)\dA(z)=0.
\end{equation}
\end{enumerate}
\end{conjecture}
For nonnegative symbols, compactness is characterized by the vanishing of fixed-radius Bergman disk averages and by the uniform vanishing of Carleson box averages; see the positive-symbol theory in~\cite{HKZ2000,Zhu1988,Zhu2007}. Angular pointwise vanishing of the latter averages is not sufficient, even for nonnegative symbols, as shown in Section 5. Zhu's report points out that Jordi Pau and K.~Zhu proved that if the Bergman disk averages tend to zero at the boundary for \emph{every} fixed $r>0$, then $T_\varphi$ is compact; unfortunately, we have not been able to find the relevant literature. In the report,  Zhu asks whether ``EVERY'' could be replaced by ``SOME.'' We obtain compactness criteria in terms of local averages over Bergman disks and Carleson boxes, yielding the following results.
\begin{theorem}\label{thm:main}
Let \(\varphi\in L^\infty(\mathbb D)\), and fix \(r_0>0\). The following five conditions are equivalent:
\begin{enumerate}[label=\textnormal{(\roman*)}]
\item $T_\varphi$ is compact on $\A$;
\item for every fixed Bergman radius $r>0$,
\begin{equation}\label{eq:bergman}
\lim_{|a|\to1}\frac1{|D(a,r)|}\int_{D(a,r)}\varphi(z)\dA(z)=0;
\end{equation}
\item there exists a sequence $r_j\downarrow0$ such that, for every fixed $j$,
\[
\lim_{|a|\to1}\frac1{|D(a,r_j)|}\int_{D(a,r_j)}\varphi(z)\dA(z)=0.
\]
\item for  some fixed radius $r_0$,
\[
\lim_{|a|\to1}\frac1{|D(a,r_0)|}\int_{D(a,r_0)}\varphi(z)\dA(z)=0.
\]
\item there exists a fixed radius $r_*>0$ such that
\[
\lim_{|a|\to1}\frac1{|D(a,r_*)|}\int_{D(a,r_*)}\varphi(z)\dA(z)=0.
\]
\end{enumerate}
Condition \textnormal{(ii)} is equivalent to requiring that, for every fixed $\rho\in(0,1)$,
\begin{equation}\label{eq:main}
\lim_{|a|\to1}\frac1{|E_\rho(a)|}\int_{E_\rho(a)}\varphi(z)\dA(z)=0.
\end{equation}
All boundary limits hold as the centers approach the unit circle in an arbitrary manner, with no restriction on the direction of approach;
condition \textnormal{(iii)} does not require uniform convergence in  $j$.
\end{theorem}
\begin{theorem}\label{thm:carleson}
Let $\varphi\in L^\infty(\D)$. 
\[
C_h\varphi(\theta)=\frac1{|S(\theta,h)|}\int_{S(\theta,h)}\varphi\dA.
\]Then $T_\varphi$ is compact on $\A$ if and only if
\begin{equation}\label{eq:boxcondition}
\lim_{h\downarrow0}\sup_{\theta\in\mathbb R}|C_h\varphi(\theta)|=0.
\end{equation}
More precisely, condition \eqref{eq:boxcondition} implies that, for every fixed $\rho\in(0,1)$,
\begin{equation}\label{eq:boximpliesdisk}
\lim_{|a|\to1}\frac1{|E_\rho(a)|}\int_{E_\rho(a)}\varphi\dA=0.
\end{equation}

\end{theorem}

\section{Preliminary lemmas}
Automorphism-induced localization is a standard tool in compactness criteria for Bergman-space Toeplitz operators; see~\cite{AZ1998,SZ1992,TV2018}.
We use the automorphism-induced unitary operator of Axler and Zheng in~\cite{AZ1998}, defined using the derivative of the disk automorphism:
$\sigma_a'=-k_a$, whereas the usual automorphism convention uses $k_a$.
This does not affect unitarity or the pairing identity below.
\begin{lemma}\label{lem:unitary}
For $a\in\D$, define
\[
(U_af)(w)=f(\sigma_a(w))\sigma_a'(w),\qquad f\in\A.
\]
Then $U_a$ is unitary, $U_a^2=I$, and, for all $f,g\in\A$,
\begin{equation}\label{eq:covariance}
\langle T_\varphi U_af,U_ag\rangle
=\int_\D\varphi(\sigma_a(w))f(w)\overline{g(w)}\dA(w).
\end{equation}
\end{lemma}
\begin{proof}
A direct calculation gives
\[
\sigma_a\circ\sigma_a=\id,
\qquad \sigma_a'(w)=-\frac{1-|a|^2}{(1-\overline a w)^2}.
\]
By the area change-of-variables formula,
\[
\norm{U_af}_2^2
=\int_\D|f(\sigma_a(w))|^2|\sigma_a'(w)|^2\dA(w)
=\int_\D|f(z)|^2\dA(z).
\]
Differentiating $\sigma_a(\sigma_a(w))=w$ yields
$\sigma_a'(\sigma_a(w))\sigma_a'(w)=1$, so $U_a^2=I$.
Thus $U_a$ is a surjective isometry and hence is unitary.

Since $U_ag\in\A$, the defining property of the orthogonal projection and the change of variables $z=\sigma_a(w)$ give
\begin{align*}
\langle T_\varphi U_af,U_ag\rangle
&=\int_\D\varphi(z)f(\sigma_a(z))
  \overline{g(\sigma_a(z))}|\sigma_a'(z)|^2\dA(z)\\
&=\int_\D\varphi(\sigma_a(w))f(w)\overline{g(w)}\dA(w).
\end{align*}
This proves \eqref{eq:covariance}.
\end{proof}

\begin{lemma}\label{lem:weak}
For every fixed analytic polynomial $p$, as $|a|\to1$,
\[
U_ap\rightharpoonup0\quad\text{in }\A.
\]
\end{lemma}
\begin{proof}
Since $U_a$ is unitary, $\norm{U_ap}_2=\norm p_2$. For every fixed $R\in(0,1)$,
\[
\sup_{|w|\le R}|U_ap(w)|
\le\norm p_{L^\infty(\D)}\frac{1-|a|^2}{(1-R)^2}\longrightarrow0.
\]
The reproducing kernel of $\A$ is $K_b(w)=(1-\overline b w)^{-2}$; see~\cite[Chapter~1]{HKZ2000}. It satisfies
$\langle f,K_b\rangle=f(b)$. Thus, for every $b\in\D$,
\[
\langle U_ap,K_b\rangle=(U_ap)(b)\longrightarrow0.
\]
The linear span of the reproducing kernels is dense in $\A$: if $f$ is orthogonal to every $K_b$,
then $f(b)=0$ for all $b\in\D$, and hence $f=0$.
Together with the uniform boundedness of $\norm{U_ap}_2$, density implies that $\langle U_ap,g\rangle\to0$ for every $g\in\A$, giving the asserted weak convergence.

\end{proof}
\begin{lemma}\label{lem:weakstar}
If $\varphi\in L^\infty(\D)$ and $T_\varphi$ is compact, then
\begin{equation}\label{eq:weakstar}
\varphi\circ\sigma_a\overset{w^*}{\longrightarrow}0
\quad\text{in }L^\infty(\D),\qquad |a|\to1.
\end{equation}
Equivalently, for every fixed $h\in L^1(\D,\mathrm dA)$,
\begin{equation}\label{eq:test}
\int_\D\varphi(\sigma_a(w))h(w)\dA(w)\longrightarrow0.
\end{equation}
\end{lemma}
\begin{proof}
For any fixed analytic polynomial $p$, Lemma~\ref{lem:weak} and compactness give
\[
\norm{T_\varphi U_ap}_2\longrightarrow0.
\]
Indeed, a compact operator maps bounded weakly null sequences to norm-null sequences.
Taking another analytic polynomial $q$ and applying Lemma~\ref{lem:unitary}, we obtain
\begin{align*}
\abs{\int_\D\varphi(\sigma_a(w))p(w)\overline{q(w)}\dA(w)}
&=|\langle T_\varphi U_ap,U_aq\rangle|\\
&\le\norm{T_\varphi U_ap}_2\norm q_2\longrightarrow0.
\end{align*}
In particular, taking $p(w)=w^m$ and $q(w)=w^n$, we see that, for every $m,n\ge0$,
\[
\int_\D\varphi(\sigma_a(w))w^m\overline w^{\,n}\dA(w)\longrightarrow0.
\]
By linearity, for every fixed mixed polynomial
$Q(w,\overline w)=\sum_{m,n}c_{mn}w^m\overline w^{\,n}$, we have
\begin{equation}\label{eq:polynomial}
\int_\D\varphi(\sigma_a(w))Q(w,\overline w)\dA(w)\longrightarrow0.
\end{equation}

The algebra of mixed polynomials contains the constants, separates the points of $\overline\D$, and is closed under complex conjugation. By the Stone--Weierstrass theorem, it is uniformly dense in $C(\overline\D)$. The complex-valued case follows by approximating the real and imaginary parts separately.
Continuous functions are dense in $L^1$ for finite Lebesgue measure. Hence mixed polynomials are dense in this $L^1$ space.

Write $M=\norm\varphi_\infty$. For each fixed $a$, the map $\sigma_a$ is a smooth diffeomorphism preserving null sets, so the pullback symbol is well defined almost everywhere and satisfies
\[
\norm{\varphi\circ\sigma_a}_\infty\le M.
\]
Given $h\in L^1$ and $\varepsilon>0$, choose a mixed polynomial $Q$ such that
$\norm{h-Q}_1<\varepsilon$. Then
\begin{align*}
\abs{\int_\D\varphi(\sigma_a(w))h(w)\dA(w)}
&\le\abs{\int_\D\varphi(\sigma_a(w))Q(w,\overline w)\dA(w)}
  +M\varepsilon.
\end{align*}
By \eqref{eq:polynomial}, the upper limit of the left-hand side as $|a|\to1$ is at most $M\varepsilon$.
Letting $\varepsilon\downarrow0$ proves \eqref{eq:test}.
\end{proof}
\section{Proof of Theorem 1.1}
\subsection{Theorem 1.1: \texorpdfstring{\textnormal{(i)}$\Rightarrow$\textnormal{(ii)}}{(i) implies (ii)}}\label{sec:necessity}
\begin{proof}
Suppose that $T_\varphi$ is compact.
The set $E_\rho(a)$ is a Euclidean disk whose center and radius are, respectively,
\[
c_{a,\rho}=\frac{(1-\rho^2)a}{1-\rho^2|a|^2},
\qquad
R_{a,\rho}=\frac{\rho(1-|a|^2)}{1-\rho^2|a|^2}.
\]
Hence
\begin{equation}\label{eq:area}
|E_\rho(a)|=R_{a,\rho}^2
=\frac{\rho^2(1-|a|^2)^2}{(1-\rho^2|a|^2)^2}.
\end{equation}
Since $E_\rho(a)=\sigma_a(\{|w|<\rho\})$, a change of variables gives
\begin{align}
\frac1{|E_\rho(a)|}\int_{E_\rho(a)}\varphi(z)\dA(z)
&=\frac1{|E_\rho(a)|}\int_{|w|<\rho}
 \varphi(\sigma_a(w))\frac{(1-|a|^2)^2}{|1-\overline a w|^4}\dA(w)\notag\\
&=\int_\D\varphi(\sigma_a(w))H_{a,\rho}(w)\dA(w),\label{eq:average}
\end{align}
where
\begin{equation}\label{eq:weight}
H_{a,\rho}(w)=\mathbf1_{\{|w|<\rho\}}
\frac{(1-\rho^2|a|^2)^2}{\rho^2|1-\overline a w|^4}.
\end{equation}
Since $H_{a,\rho}$ depends on the center $a$, this expression cannot yet be combined with Lemma~\ref{lem:weakstar}; we therefore freeze the test function by passing to a fixed limit 
$H_{\zeta,\rho}.$

Suppose, to the contrary, that the average over $E_\rho(a)$ does not tend to zero as $|a|\to1$ for the fixed $\rho$ chosen above. Then there exist $\varepsilon_0>0$ and a sequence $a_j\in\D$
 such that $|a_j|\to1$ and
\begin{equation}\label{eq:contradiction}
\abs{\frac1{|E_\rho(a_j)|}\int_{E_\rho(a_j)}\varphi\dA}
\ge\varepsilon_0\qquad\text{for all }j.
\end{equation}
By compactness of the closed disk, we may pass to a subsequence and assume that $a_j\to\zeta$, where $|\zeta|=1$.
Define the fixed test function
\[
H_{\zeta,\rho}(w)=\mathbf1_{\{|w|<\rho\}}
\frac{(1-\rho^2)^2}{\rho^2|1-\overline\zeta w|^4}.
\]
For $|w|<\rho$,
\[
|1-\overline a_j w|\ge1-\rho,
\qquad |1-\overline\zeta w|\ge1-\rho.
\]
Consequently, $H_{a_j,\rho}$ converges pointwise to $H_{\zeta,\rho}$, and
\[
0\le H_{a_j,\rho}(w),\ H_{\zeta,\rho}(w)
\le\frac{\mathbf1_{\{|w|<\rho\}}}{\rho^2(1-\rho)^4}.
\]
The right-hand side is integrable, so the dominated convergence theorem yields
\begin{equation}\label{eq:L1}
\norm{H_{a_j,\rho}-H_{\zeta,\rho}}_1\longrightarrow0.
\end{equation}
In particular, $H_{\zeta,\rho}\in L^1(\D)$.

By \eqref{eq:average},
\begin{align*}
\abs{\frac1{|E_\rho(a_j)|}\int_{E_\rho(a_j)}\varphi\dA}
&\le\abs{\int_\D\varphi(\sigma_{a_j}(w))H_{\zeta,\rho}(w)\dA(w)}\\
&\quad+\norm\varphi_\infty\norm{H_{a_j,\rho}-H_{\zeta,\rho}}_1.
\end{align*}
The first term tends to zero by Lemma~\ref{lem:weakstar}, and the second by \eqref{eq:L1}, contradicting \eqref{eq:contradiction}.
Thus \eqref{eq:main} holds. Taking $\rho=\tanh r$ gives \eqref{eq:bergman}.
\end{proof}
\subsection{Theorem 1.1: \texorpdfstring{\textnormal{(ii)}$\Rightarrow$\textnormal{(iii)}$\Rightarrow$\textnormal{(i)}}{(ii) implies (iii) implies (i)}}
\label{sec:approx}
For $0<\rho<1$, define
\[
M_\rho\varphi(a)=\frac1{|E_\rho(a)|}\int_{E_\rho(a)}\varphi(w)\dA(w).
\]
The averaging kernel is measurable, so $M_\rho\varphi$ is measurable, and
$\norm{M_\rho\varphi}_\infty\le\norm\varphi_\infty$.
In this subsection, the estimates concern the local oscillation of analytic test functions. Throughout, $C$ denotes an absolute constant that may change from line to line but is independent of $\rho\in(0,1/32]$, the center, and the functions involved. We need the following lemmas to complete the proof.

\begin{lemma}\label{lem:mass}
Let $\delta(w)=1-|w|^2$, and set
\[
q_\rho(z,w)=\frac{\mathbf1_{E_\rho(z)}(w)}{|E_\rho(z)|},
\qquad m_\rho(w)=\int_\D q_\rho(z,w)\dA(z).
\]
Then $\int_\D q_\rho(z,w)\dA(w)=1$. Moreover, for $0<\rho\le1/32$,
\begin{equation}\label{eq:mass}
\sup_{w\in\D}|m_\rho(w)-1|\le C\rho,
\qquad 0\le m_\rho(w)\le2.
\end{equation}
\end{lemma}
\begin{proof}
The first identity follows immediately from the definition. By symmetry of the pseudohyperbolic distance,
\[
m_\rho(w)=\int_{E_\rho(w)}\frac{\dA(z)}{|E_\rho(z)|}.
\]
If $z\in E_\rho(w)$, write $z=\sigma_w(u)$ with $|u|<\rho$. Then
\[
\frac{\delta(z)}{\delta(w)}=\frac{1-|u|^2}{|1-\overline w u|^2},
\qquad
\frac{1-\rho}{1+\rho}\le\frac{\delta(z)}{\delta(w)}
\le\frac{1+\rho}{1-\rho}.
\]
By the area formula \eqref{eq:area},
\[
\frac{|E_\rho(w)|}{|E_\rho(z)|}
=\left(\frac{\delta(w)}{\delta(z)}\right)^2
 \left(\frac{1-\rho^2|z|^2}{1-\rho^2|w|^2}\right)^2.
\]
The second ratio on the right lies in $[1-\rho^2,(1-\rho^2)^{-1}]$, so
\[
(1-\rho)^4\le\frac{|E_\rho(w)|}{|E_\rho(z)|}\le(1-\rho)^{-4}.
\]
Taking the normalized area average over $z\in E_\rho(w)$ gives
\[
(1-\rho)^4\le m_\rho(w)\le(1-\rho)^{-4}.
\]
For $\rho\le1/32$, this implies \eqref{eq:mass}.
\end{proof}

\begin{lemma}\label{lem:oscillation}
For $f\in\A$, define the nonnegative function
\[
B_f(w)^2=\frac1{\delta(w)^2}
\int_{|u-w|<\delta(w)/4}|f(u)|^2\dA(u).
\]
Then
\begin{equation}\label{eq:Bglobal}
\int_\D B_f(w)^2\dA(w)\le C\norm f_2^2.
\end{equation}
If $0<\rho\le1/32$ and $z\in E_\rho(w)$, then
\begin{align}
|f(z)|+|f(w)|&\le C B_f(w),\label{eq:Bpoint}\\
|f(z)-f(w)|&\le C\rho B_f(w).\label{eq:Bosc}
\end{align}
\end{lemma}
\begin{proof}
Since $\delta(w)/4<1-|w|$, the Euclidean disk in this definition is contained in $\D$.
Set
\[
h(\xi)=f\left(w+\frac{\delta(w)}4\xi\right),\qquad |\xi|<1.
\]
A change of variables gives $\norm h_2^2=16B_f(w)^2$. If
$h(\xi)=\sum_{n\ge0}c_n\xi^n$, orthogonality of the monomials in the Bergman space gives
$\norm h_2^2=\sum_{n\ge0}|c_n|^2/(n+1)$.
By the Cauchy--Schwarz inequality,
\begin{align*}
|h(\xi)|&\le\norm h_2
 \left(\sum_{n\ge0}(n+1)|\xi|^{2n}\right)^{1/2},\\
|h'(\xi)|&\le\norm h_2
 \left(\sum_{n\ge1}n^2(n+1)|\xi|^{2n-2}\right)^{1/2}.
\end{align*}
Both series are uniformly bounded for $|\xi|\le1/2$. Hence
\begin{equation}\label{eq:derivative}
\sup_{|v-w|\le\delta(w)/8}
\bigl(|f(v)|+\delta(w)|f'(v)|\bigr)\le C B_f(w).
\end{equation}
If $z=\sigma_w(u)$ with $|u|<\rho$, then
\[
|z-w|=\frac{|u|\delta(w)}{|1-\overline w u|}
\le\frac{\rho}{1-\rho}\delta(w)\le2\rho\delta(w).
\]
When $\rho\le1/32$, the segment $[w,z]$ lies in the disk appearing in \eqref{eq:derivative}.
Integrating the derivative along this segment gives \eqref{eq:Bosc}, while the pointwise estimate gives \eqref{eq:Bpoint}.

By Tonelli's theorem~\cite[Theorem~5.28]{Axler2020},
\[
\int_\D B_f(w)^2\dA(w)
=\int_\D|f(u)|^2
\left[\int_{\{|u-w|<\delta(w)/4\}}\frac{\dA(w)}{\delta(w)^2}\right]\dA(u).
\]
If $|u-w|<\delta(w)/4$, then
\[
|\delta(u)-\delta(w)|\le2|u-w|<\delta(w)/2,
\]
Thus $\tfrac23\delta(u)<\delta(w)<2\delta(u)$ and
$|u-w|<\delta(u)/2$. Therefore the integral in brackets is at most
\[
\frac9{4\delta(u)^2}\int_{|w-u|<\delta(u)/2}\dA(w)
\le\frac9{16}.
\]
This proves \eqref{eq:Bglobal}.
\end{proof}

\begin{lemma}\label{lem:approx}
There is an absolute constant $C>0$ such that, for every $\varphi\in L^\infty(\D)$ and
every $0<\rho\le1/32$,
\begin{equation}\label{eq:normapprox}
\norm{T_\varphi-T_{M_\rho\varphi}}\le C\rho\norm\varphi_\infty.
\end{equation}
\end{lemma}
\begin{proof}
Take arbitrary $f,g\in\A$, and write $F(z)=f(z)\overline{g(z)}$.
By Lemma~\ref{lem:oscillation}, when $z\in E_\rho(w)$,
\begin{align}
|F(z)-F(w)|
&\le |f(z)-f(w)|\,|g(z)|+|f(w)|\,|g(z)-g(w)|\notag\\
&\le C\rho B_f(w)B_g(w).\label{eq:productosc}
\end{align}
The definition of the averaging kernel gives
\[
\langle T_{M_\rho\varphi}f,g\rangle
=\int_\D\varphi(w)\left[\int_\D q_\rho(z,w)F(z)\dA(z)\right]\dA(w).
\]
Fubini's theorem  applies, since
\[
\int_\D\int_\D |\varphi(w)|q_\rho(z,w)|F(z)|\dA(w)\dA(z)
\le\norm\varphi_\infty\norm f_2\norm g_2<\infty .
\]
We therefore obtain the decomposition
\begin{align}
\langle(T_{M_\rho\varphi}-T_\varphi)f,g\rangle
={}&\int_\D\varphi(w)\int_\D q_\rho(z,w)
 [F(z)-F(w)]\dA(z)\dA(w)\notag\\
&+\int_\D\varphi(w)[m_\rho(w)-1]F(w)\dA(w).
\label{eq:spliterror}
\end{align}
By \eqref{eq:productosc}, the bound $m_\rho\le2$, and \eqref{eq:Bglobal}, the absolute value of the first term is at most
\[
C\rho\norm\varphi_\infty\int_\D B_f(w)B_g(w)\dA(w)
\le C\rho\norm\varphi_\infty\norm f_2\norm g_2.
\]
By \eqref{eq:mass}, the absolute value of the second term is at most
\[
C\rho\norm\varphi_\infty\int_\D|f(w)g(w)|\dA(w)
\le C\rho\norm\varphi_\infty\norm f_2\norm g_2.
\]
Taking the supremum over unit vectors $f,g$ gives \eqref{eq:normapprox}.
\end{proof}
\begin{lemma}\label{lem:vanishing}
Let $\psi\in L^\infty(\mathbb D)$. If
\[
\lim_{R\uparrow1}
\operatorname*{ess\,sup}_{R<|z|<1}|\psi(z)|=0,
\]
then $T_\psi$ is compact on $A^2(\mathbb D)$.
\end{lemma}
\begin{proof}
For $R<1$, set $\psi_R=\psi\mathbf1_{\{|w|\le R\}}$.
Consider the integral operator on $L^2(\D)$ with kernel
\[
L_R(z,w)=\frac{\psi_R(w)}{(1-z\overline w)^2}.
\]
By the formula for the norm of the reproducing kernel,
\begin{align*}
\int_\D\int_\D |L_R(z,w)|^2\dA(z)\dA(w)
&=\int_{|w|\le R}\frac{|\psi(w)|^2}{(1-|w|^2)^2}\dA(w)<\infty.
\end{align*}
A square-integrable kernel defines a Hilbert--Schmidt operator and hence a compact operator. The restriction of this integral operator to $\A$ is precisely $T_{\psi_R}$,
and its range lies in the closed subspace $\A$. Hence $T_{\psi_R}:\A\to\A$ is compact.
Moreover,
\[
\norm{T_\psi-T_{\psi_R}}
\le\norm{\psi\mathbf1_{\{|z|>R\}}}_\infty\longrightarrow0.
\]
Since the compact operators are closed in the operator norm~\cite{Conway1990}, $T_\psi$ is compact.
\end{proof}

\begin{proof}[Proof of \textnormal{(ii)}$\Rightarrow$\textnormal{(iii)}$\Rightarrow$\textnormal{(i)} in Theorem 1.1]
\textbf{First, we prove \textnormal{(ii)}$\Rightarrow$\textnormal{(iii)}.}
If the boundary vanishing condition holds for every fixed radius, take $r_j=1/j$ to obtain condition \textnormal{(iii)}.

\textbf{Next, we prove \textnormal{(iii)}$\Rightarrow$\textnormal{(i)}.}
Suppose that $r_j\downarrow0$ satisfies condition \textnormal{(iii)}, and set
\[
\rho_j=\tanh r_j\downarrow0,\qquad
\psi_j=M_{\rho_j}\varphi.
\]
Since $D(a,r_j)=E_{\rho_j}(a)$, condition \textnormal{(iii)} implies that, for every fixed $j$,
\[
\psi_j(a)\longrightarrow0\quad(|a|\to1),
\qquad \norm{\psi_j}_\infty\le\norm\varphi_\infty.
\]
By Lemma~\ref{lem:vanishing}, each $T_{\psi_j}$ is compact.
For sufficiently large $j$, we have $\rho_j\le1/32$. Hence Lemma~\ref{lem:approx} gives
\[
\norm{T_\varphi-T_{\psi_j}}\le C\rho_j\norm\varphi_\infty\longrightarrow0.
\]
The compact operators are closed in the operator norm, so $T_\varphi$ is compact, proving condition \textnormal{(i)}.
\end{proof}
To prove \textnormal{(v)}$\Rightarrow$\textnormal{(ii)} in Theorem 1.1, we need some further preparation.
Write $\mathbb H=\mathbb R\times(0,\infty)$, equipped in this subsection with ordinary area measure $\dd x\dd t$.
\begin{lemma}\label{lem:halfplane-unique}
Fix $0<q<1$. If $F\in L^\infty(\mathbb H)$ satisfies
\begin{equation}\label{eq:halfplane-zero}
\int_{B((x,t),qt)}F(u,v)\dd u\dd v=0
\qquad(x\in\mathbb R,\ t>0),
\end{equation}
then $F=0$ almost everywhere. Here
\[
B((x,t),qt)
=
\left\{(u,v)\in\mathbb{R}^2:
(u-x)^2+(v-t)^2<(qt)^2
\right\}
\]
denotes the Euclidean disk centered at $(x,t)$ with radius $qt$.
Since $0<q<1$, this disk is contained in the upper half-plane
$\mathbb H=\mathbb R\times(0,\infty)$. Consequently,
\begin{equation}\label{eq:halfplane-dense}
\overline{\Span\{\mathbf1_{B((x,t),qt)}:x\in\mathbb R,\ t>0\}}^{\,L^1(\mathbb H)}
=L^1(\mathbb H).
\end{equation}
\end{lemma}
\begin{proof}
We write $\mathcal S(\mathbb R)$ for the Schwartz space of
smooth functions $\vartheta$ such that
\[
\sup_{x\in\mathbb R}
(1+|x|)^m |\vartheta^{(n)}(x)|<\infty
\qquad (m,n\in\mathbb N_0),
\]
and $\mathcal S'(\mathbb R)$ for its continuous dual,
the space of tempered distributions. Choose a Schwartz function $\psi$ such that $\int\psi=1$ and $\widehat\psi\in C_c^\infty(\mathbb R)$.
Let $\psi_\varepsilon(x)=\varepsilon^{-1}\psi(x/\varepsilon)$, and define
\[
F_\varepsilon(x,t)=\int_{\mathbb R}\psi_\varepsilon(b)F(x-b,t)\dd b.
\]
Next, choose a nonnegative function $\eta_\delta\in C_c^\infty(-\delta,\delta)$ with $\int\eta_\delta=1$, and set
\begin{equation}\label{eq:single-regularize}
G(x,t)=\int_{\mathbb R}\eta_\delta(s)F_\varepsilon(e^sx,e^st)\dd s.
\end{equation}
For the limiting argument at the end of the proof, these functions are chosen from a single scaling family, namely
\[
\eta_\delta(s)=\delta^{-1}\eta(s/\delta),
\]
where $\eta\in C_c^\infty(-1,1)$ is nonnegative and
$\int_{\mathbb R}\eta(s)\dd s=1$. Thus $\{\eta_\delta\}_{\delta>0}$ is an
approximate identity on $\mathbb R$.
All integrals are justified by boundedness and the integrability of the kernels, and Fubini's theorem applies.
Horizontal translations preserve the family of disks, and dilations satisfy
$e^s B((x,t),qt)=B((e^sx,e^st),qe^st)$.
Hence both $F_\varepsilon$ and $G$ satisfy \eqref{eq:halfplane-zero}.
The dilation change of variables introduces only the nonzero factor $e^{-2s}$, which does not affect the vanishing of the integrals.

For completeness, the translation step follows from Fubini's theorem:
\begin{align*}
\int_{B((x,t),qt)}F_\varepsilon(u,v)\dd u\dd v
&=\int_{\mathbb R}\psi_\varepsilon(b)
  \int_{B((x,t),qt)}F(u-b,v)\dd u\dd v\dd b\\
&=\int_{\mathbb R}\psi_\varepsilon(b)
  \int_{B((x-b,t),qt)}F(y,v)\dd y\dd v\dd b=0.
\end{align*}
Likewise, using $(y,v)=(e^su,e^sv)$, we have
\begin{align*}
\int_{B((x,t),qt)}G(u,v)\dd u\dd v
&=\int_{\mathbb R}\eta_\delta(s)e^{-2s}
  \int_{B((e^sx,e^st),qe^st)}F_\varepsilon(y,v)\dd y\dd v\dd s=0.
\end{align*}
The absolute convergence needed for these applications of Fubini follows from
$F\in L^\infty(\mathbb H)$, $\psi_\varepsilon\in L^1(\mathbb R)$, and the
compact support of $\eta_\delta$.

Fix $\varepsilon,\delta>0$ from now on. Constants in the
regularity estimates below may depend on these parameters,
the chosen kernels, and $\norm F_\infty$, but not on $x$ or $t$.
Write
\[
H_j=\partial_x^jF_\varepsilon
     =(\partial_x^j\psi_\varepsilon)*_xF,
\qquad j\ge0.
\]
These horizontal derivatives are well defined distributionally,
with bounded measurable representatives satisfying
\[
\norm{H_j}_\infty
\le \norm F_\infty
   \norm{\partial_x^j\psi_\varepsilon}_{L^1(\mathbb R)}.
\]

We first verify that $G$ has a smooth representative.
Changing variables $v=e^st$ in
\eqref{eq:single-regularize} gives
\begin{equation}\label{eq:single-smooth-kernel}
G(x,t)
=
\int_0^\infty\int_{\mathbb R}
F(y,v)\,
\eta_\delta\!\left(\log\frac vt\right)
\psi_\varepsilon\!\left(\frac vt x-y\right)
\dd y\,\frac{\dd v}{v}.
\end{equation}
On each compact subset of $\mathbb H$, the $v$-integration
is restricted to a fixed compact subinterval of $(0,\infty)$.
Moreover, every derivative in $(x,t)$ of the kernel in
\eqref{eq:single-smooth-kernel} is dominated there by an
integrable function of $(y,v)$: the coefficients are bounded
on that compact set, and the horizontal kernels and all their
derivatives are Schwartz functions with uniformly bounded
translations. Since $F$ is bounded, differentiation under the
integral is valid to every order. Thus $G\in C^\infty(\mathbb H)$.

For the quantitative estimates, let
\[
D=x\partial_x+t\partial_t.
\]
Horizontal differentiation and integration by parts in the
scale variable, initially in the distributional sense, give
\begin{align*}
G_x
&=\int_{\mathbb R}\eta_\delta(s)e^s
       H_1(e^sx,e^st)\dd s,\\
G_{xx}
&=\int_{\mathbb R}\eta_\delta(s)e^{2s}
       H_2(e^sx,e^st)\dd s,\\
DG
&=-\int_{\mathbb R}\eta_\delta'(s)
       H_0(e^sx,e^st)\dd s,\\
D^2G
&=\int_{\mathbb R}\eta_\delta''(s)
       H_0(e^sx,e^st)\dd s,\\
D(G_x)
&=-\int_{\mathbb R}
       \bigl(\eta_\delta'(s)+\eta_\delta(s)\bigr)e^s
       H_1(e^sx,e^st)\dd s.
\end{align*}
Consequently,
\[
|G|+|G_x|+|G_{xx}|+|DG|+|D^2G|+|D(G_x)|
\le C.
\]
The identities
\[
tG_t=DG-xG_x,\qquad
tG_{xt}=D(G_x)-xG_{xx}
\]
and
\[
t^2G_{tt}
=D^2G-DG-2xD(G_x)+x^2G_{xx}
\]
therefore imply
\begin{equation}\label{eq:single-derivatives}
|G(x,t)|+|G_x(x,t)|\le C,
\qquad
|G_t(x,t)|\le \frac{C(1+|x|)}{t},
\end{equation}
as well as
\begin{equation}\label{eq:single-second-derivatives}
|G_{xx}(x,t)|\le C,\qquad
|G_{xt}(x,t)|\le \frac{C(1+|x|)}{t},\qquad
|G_{tt}(x,t)|\le \frac{C(1+x^2)}{t^2}.
\end{equation}
In particular, on every strip $0<a\le t\le b<\infty$,
all derivatives of total order at most two are bounded by
$C_{a,b}(1+x^2)$.

Finally, choose $L<\infty$ such that
$\operatorname{supp}\widehat{\psi_\varepsilon}\subset[-L,L]$.
Horizontal convolution gives this Fourier support bound for
$F_\varepsilon(\cdot,t)$ at almost every height.
Since $s\in\operatorname{supp}\eta_\delta\subset[-\delta,\delta]$,
horizontal dilation enlarges the support by at most $e^\delta$.
It follows that
\[
\operatorname{supp}\mathcal F_xG(\cdot,t)
\subset[-\Lambda,\Lambda],
\qquad \Lambda=e^\delta L,
\]
for every $t>0$. Here Fourier transforms are understood in
the sense of tempered distributions. Indeed, the uniform bound on $G$ and its pointwise continuity
imply that, for every $\vartheta\in\mathcal S(\mathbb R)$,
\[
t\longmapsto
\int_{\mathbb R}G(x,t)\vartheta(x)\dd x
\]
is continuous by dominated convergence. Hence
$G(\cdot,t)$ is continuous in the weak topology of
$\mathcal S'(\mathbb R)$. Since the Fourier transform is continuous
on $\mathcal S'(\mathbb R)$, the common Fourier support bound
extends from almost every height to every height by testing
$\mathcal F_xG(\cdot,t)$ against functions in
$C_c^\infty(\mathbb R\setminus[-\Lambda,\Lambda])$.

For fixed $\varepsilon$, let $\delta\downarrow0$. Continuity of dilations in local $L^1$ gives
$G\to F_\varepsilon$ in $L^1_{\mathrm{loc}}$. Then, as $\varepsilon\downarrow0$,
$F_\varepsilon\to F$ in $L^1_{\mathrm{loc}}$.
More explicitly, for every compact set $K\Subset\mathbb H$, the choice of
$\eta_\delta$ above and the continuity of dilations in $L^1(K)$ imply
\[
G\longrightarrow F_\varepsilon\quad\text{in }L^1(K)\qquad(\delta\downarrow0).
\]
The standard approximation-of-the-identity property of horizontal convolution
similarly gives
\[
F_\varepsilon\longrightarrow F\quad\text{in }L^1(K)\qquad(\varepsilon\downarrow0).
\]
Consequently, once the vanishing of $G$ has been established for every fixed
$\varepsilon,\delta>0$, these two local $L^1$ convergences imply $F=0$ almost everywhere.
Thus it suffices to prove that each $G$ vanishes.

\medskip
Consider the Hilbert space
\[
X=L^2(\mathbb R,w(x)\dd x),\qquad w(x)=(1+x^2)^{-4}.
\]
Then
\[
|\langle v,\vartheta\rangle|
\le\|v\|_X\left(\int|\vartheta(x)|^2(1+x^2)^4dx\right)^{1/2}, \quad v \in X,
\quad\vartheta\in\mathcal S,
\]
Bounded functions and functions of polynomial growth of degree at most two belong to $X$.
For $v\in X$ and $h\in\mathbb R$, define the horizontal
translation $\tau_hv$ by
\[
(\tau_hv)(x)=v(x-h).
\]
Since
\[
1+u^2\le 2\bigl(1+(u+h)^2\bigr)(1+h^2),
\]
we have
\[
w(u+h)\le 16(1+h^2)^4w(u).
\]
Consequently,
\[
\begin{aligned}
\|\tau_hv\|_X^2
&=\int_{\mathbb R}|v(u)|^2w(u+h)\,du\\
&\le 16(1+h^2)^4\|v\|_X^2,
\end{aligned}
\]
and hence
\[
\|\tau_hv\|_X
\le 4(1+h^2)^2\|v\|_X
\le 4(1+|h|)^4\|v\|_X.
\]
Choose $\chi\in C_c^\infty(\mathbb R)$ that is identically $1$ on a neighborhood of $[-\Lambda,\Lambda]$.
With the convention $\widehat v(\xi)=\int v(x)e^{-ix\xi}\dd x$, set
\[
K=\mathcal F^{-1}(-\xi^2\chi(\xi)),\qquad Bv=K*v.
\]
The function $K$ is Schwartz. Minkowski's inequality and the translation estimate above give
\[
\norm{Bv}_X\le
C\left(\int|K(h)|(1+|h|)^4\dd h\right)\norm v_X.
\]
Thus $B$ is bounded on $X$. If the horizontal Fourier support of $v$ is contained in
$[-\Lambda,\Lambda]$, then $Bv=\partial_x^2v$, initially in the distributional sense,
and also in $X$ for the smooth functions used below. 

\medskip
For $0\le r<t$, define
\[
U(x,t,r)=\frac1\pi\int_{a^2+b^2<1}G(x+ra,t+rb)\dd a\dd b.
\]
For each $(t,r)$, we regard $U(\cdot,t,r)$ as an element of $X$ and,
when no confusion can arise, write it simply as $U(t,r)$. All derivatives of
$U$ below are understood in this $X$-valued sense.
For $r>0$, this is the ordinary disk average of radius $r$,
and $U(x,t,0)=G(x,t)$.
For the regularity argument, use the same formula to extend
$U$ to the open parameter set
\[
\mathcal O=\{(t,r)\in\mathbb R^2:t>0,\ |r|<t\}.
\]
We claim that
\[
(t,r)\longmapsto U(\cdot,t,r)
\]
belongs to $C^2(\mathcal O;X)$, where derivatives and continuity
are taken in the norm of $X$.

Let $\mathcal K\Subset\mathcal O$. There exist constants
$a_{\mathcal K},b_{\mathcal K},L_{\mathcal K}>0$ such that,
for $(t,r)\in\mathcal K$ and $a^2+b^2\le1$,
\[
a_{\mathcal K}\le t+rb\le b_{\mathcal K},
\qquad |ra|\le L_{\mathcal K}.
\]
Thus \eqref{eq:single-derivatives} and
\eqref{eq:single-second-derivatives} show that all parameter
derivatives of the integrand of total order at most two
are bounded in absolute value by $C_{\mathcal K}(1+x^2)$.
Indeed, if
\[
\mathcal A H(x,t,r)
=\frac1\pi\int_{a^2+b^2<1}
H(x+ra,t+rb)\dd a\dd b,
\]
then the pointwise derivative formulas are
\begin{align*}
U_t&=\mathcal A(G_t),\\
U_r&=\frac1\pi\int_{a^2+b^2<1}
       (aG_x+bG_t)(x+ra,t+rb)\dd a\dd b,\\
U_{tt}&=\mathcal A(G_{tt}),\\
U_{tr}&=\frac1\pi\int_{a^2+b^2<1}
       (aG_{xt}+bG_{tt})(x+ra,t+rb)\dd a\dd b,\\
U_{rr}&=\frac1\pi\int_{a^2+b^2<1}
       (a^2G_{xx}+2abG_{xt}+b^2G_{tt})
       (x+ra,t+rb)\dd a\dd b.
\end{align*}

The dominating function belongs to $X$, since
\[
\int_{\mathbb R}(1+x^2)^2w(x)\dd x
=
\int_{\mathbb R}(1+x^2)^{-2}\dd x<\infty.
\]
The smoothness of $G$ and dominated convergence therefore
give continuity in $X$ of all the displayed expressions.
Applying the same domination to parameter difference
quotients, using the fundamental theorem of calculus on
a slightly larger compact subset of $\mathcal O$, proves
that these expressions are the strong first and second
derivatives of $U$. Hence $U\in C^2(\mathcal O;X)$.

The horizontal Fourier support of $U(\cdot,t,r)$ remains
contained in $[-\Lambda,\Lambda]$: horizontal translations
do not enlarge Fourier support, and the averaging only
combines functions having this common support.
Moreover,
\[
U_{xx}(x,t,r)
=\frac1\pi\int_{a^2+b^2<1}
G_{xx}(x+ra,t+rb)\dd a\dd b
\]
belongs to $X$. Consequently,
\[
BU(\cdot,t,r)=U_{xx}(\cdot,t,r)
\]
as an equality in $X$. Applying \eqref{eq:halfplane-zero} to $G$ at the center $(x,t)$ and radius
$qt$ gives, for every $t>0$,
\begin{equation}\label{eq:single-movingboundary}
U(t,qt)=0.
\end{equation}
Disk averages satisfy
\begin{equation}\label{eq:single-darboux}
U_{rr}+\frac3rU_r=U_{xx}+U_{tt}.
\end{equation}
Let $S(x,t,r)$ be the circular average of $G$. The divergence theorem gives
$U_{xx}+U_{tt}=2S_r/r$.
Since $U=2r^{-2}\int_0^r sS(x,t,s)\dd s$, we obtain
$S=U+(r/2)U_r$, which yields \eqref{eq:single-darboux} upon substitution.
Consequently, in $X$,
\begin{equation}\label{eq:single-wave}
U_{tt}=U_{rr}+\frac3rU_r-BU.
\end{equation}
The identity $U_r(t,0)=0$ follows from the vanishing of the first moments of the unit disk.

\medskip
The weight $r^3$ (i.e.\ $r^{n+1}$ with $n=2$) is chosen so that the radial part $U_{rr}+\tfrac3rU_r$ is  symmetric on $L^2(r^3\,\dd r)$,
which eliminates the boundary terms in the integration by parts below.
Define the nonnegative energy
\begin{equation}\label{eq:single-energy}
E(t)=\frac12\int_0^{qt}
\bigl(\norm{U_t(t,r)}_X^2+\norm{U_r(t,r)}_X^2+\norm{U(t,r)}_X^2\bigr)r^3\dd r.
\end{equation}
On every compact interval $0<a\le t\le b$, the preceding regularity justifies differentiation of the energy and integration by parts.
To make the behavior at the lower endpoint explicit, for $0<\alpha<qt$ first
set
\[
E_\alpha(t)=\frac12\int_\alpha^{qt}
\bigl(\norm{U_t(t,r)}_X^2+\norm{U_r(t,r)}_X^2+\norm{U(t,r)}_X^2\bigr)r^3\dd r.
\]
Differentiate $E_\alpha$ and integrate by parts in $r$. The terms at
$r=\alpha$ are bounded by a constant times $\alpha^3$ on $[a,b]$ and hence
vanish as $\alpha\downarrow0$.
The same $C^2(\mathcal O;X)$ regularity implies that $U_t$ and
$U_r$ remain bounded in $X$ as $r\downarrow0$, uniformly for
$t\in[a,b]$. Hence
\[
r^3\operatorname{Re}\langle U_t(t,r),U_r(t,r)\rangle_X
\longrightarrow0
\qquad(r\downarrow0),
\]
uniformly for $t\in[a,b]$. Let $\alpha\downarrow0$.
By \eqref{eq:single-wave},
\begin{align*}
E'(t)={}&(qt)^3\operatorname{Re}\langle U_t(t,qt),U_r(t,qt)\rangle_X\\
&+\frac q2(qt)^3\bigl(\norm{U_t(t,qt)}_X^2+
\norm{U_r(t,qt)}_X^2+\norm{U(t,qt)}_X^2\bigr)\\
&+\int_0^{qt}\operatorname{Re}\langle U_t,(I-B)U\rangle_Xr^3\dd r.
\end{align*}
Differentiating along the boundary in \eqref{eq:single-movingboundary} gives
$U_t(t,qt)+qU_r(t,qt)=0$.
Hence the first two lines equal
\[
-\frac{q(1-q^2)}2(qt)^3\norm{U_r(t,qt)}_X^2\le0.
\]
For the last term, the boundedness of $B$ on $X$ and Cauchy--Schwarz give
\begin{align*}
\left|\operatorname{Re}\langle U_t,(I-B)U\rangle_X\right|
&\le (1+\norm B)\norm{U_t}_X\norm U_X\\
&\le \frac{1+\norm B}{2}
   \bigl(\norm{U_t}_X^2+\norm U_X^2\bigr),
\end{align*}
where we used $2ab\leq a^2+b^2$. Since the integral of
$\norm{U_t}_X^2+\norm U_X^2$ against $r^3\dd r$ is bounded by
$2E(t)$, we may set
\[
C_0:=1+\norm{B}_{\mathcal B(X)}.
\]
Thus
\begin{equation}\label{eq:single-gronwall}
E'(t)\leq C_0 E(t),\qquad t>0.
\end{equation}
For fixed regularization parameters, $C_0$ is independent of $t$ and of the
endpoints of the interval on which the energy identity is justified.
\medskip

For $0<t\le1$ and $0\le r\le qt$, the sampling points satisfy
$t+rb\ge(1-q)t$ and $|x+ra|\le|x|+1$. By
\eqref{eq:single-derivatives},
\[
\norm{U(t,r)}_X\le C,
\qquad
\norm{U_t(t,r)}_X+\norm{U_r(t,r)}_X\le C/t.
\]
Here $U_r$ is the average of $aG_x+bG_t$, and the functions of at most
linear growth in $x$ that occur here belong to $X$. Therefore
\[
0\le E(t)\le C\int_0^{qt}(1+t^{-2})r^3\dd r
\le C(t^4+t^2)\longrightarrow0
\qquad(t\downarrow0).
\]
For fixed $T>0$ and every $0<s<T$, \eqref{eq:single-gronwall} gives
\[
E(T)\le e^{C_0(T-s)}E(s).
\]
Letting $s\downarrow0$ and using $E(s)\to0$ yields $E(T)=0$. Since the energy contains the term $\norm U_X^2$, continuity gives $U(t,r)=0$ for $0<r<qt$.
More precisely, $U(\cdot,T,r)$ is continuous as an $X$-valued function of
$r$, and $E(T)=0$ implies $U(\cdot,T,r)=0$ for every $0\le r\le qT$.

Letting $r\downarrow0$, we obtain $G(\cdot,t)=0$ in $X$.
Since $G$ is smooth, $G=0$ on $\mathbb H$.
Letting first $\delta\downarrow0$ and then $\varepsilon\downarrow0$, we obtain $F=0$ almost everywhere.
\medskip

If \eqref{eq:halfplane-dense} were false, the Hahn--Banach theorem would
produce a nonzero continuous linear functional on $L^1(\mathbb H)$ that
annihilates every disk indicator. Since
$L^1(\mathbb H)^*=L^\infty(\mathbb H)$, this functional is represented by a
nonzero $F\in L^\infty(\mathbb H)$. Hence
\[
\int_{B((x,t),qt)}F(u,v)\dd u\dd v=0
\qquad(x\in\mathbb R,\ t>0),
\]
which contradicts the uniqueness just proved.
\end{proof}
\begin{lemma}\label{lem:single-blowup}
Let $\varphi\in L^\infty(\D)$, and suppose that a fixed $\rho_0\in(0,1)$ satisfies
$M_{\rho_0}\varphi(a)\to0$ as $|a|\to1$.
For arbitrary $h_n\downarrow0$ and $|\zeta_n|=1$, define
\[
F_n(x,t)=
\begin{cases}
\varphi\bigl(\zeta_n(1-h_nt+ih_nx)\bigr),
 & |1-h_nt+ih_nx|<1,\\
0,&\text{otherwise},
\end{cases}\qquad (x,t)\in\mathbb H.
\]
Then every weak-star convergent subsequence of $F_n$ has zero limit.
\end{lemma}
\begin{proof}
Uniform boundedness and the separability of $L^1(\mathbb H)$ yield a weak-star convergent subsequence.
This is the metrizable bounded-ball case of the Banach--Alaoglu theorem.
Continue to denote the subsequence by $F_n$, and let its limit be $F$.
Every compact set $K\Subset\mathbb H$ eventually lies in the valid coordinate region, since
\[
|1-h_nt+ih_nx|<1\quad\Longleftrightarrow\quad
2t>h_n(t^2+x^2).
\]
Set
\[
c_0=\frac{1+\rho_0^2}{1-\rho_0^2},\qquad
 d_0=\frac{2\rho_0}{1-\rho_0^2}.
\]
Fix arbitrary $X\in\mathbb R$ and $Y>0$, and take
$b_n=\zeta_n(1-h_nY+ih_nX)$.
Then $b_n\in\D$ for all sufficiently large $n$, and $|b_n|\to1$.
Let $\Omega_n(X,Y)$ be the preimage of $E_{\rho_0}(b_n)$ under the affine coordinates
$z=\zeta_n(1-h_nt+ih_nx)$.

Set $w_n=1-h_nY+ih_nX$. After rotation, the Euclidean center and radius of the pseudohyperbolic disk are
\[
C_n=\frac{(1-\rho_0^2)w_n}{1-\rho_0^2|w_n|^2},\qquad
D_n=\frac{\rho_0(1-|w_n|^2)}{1-\rho_0^2|w_n|^2}.
\]
Since $|w_n|^2=1-2h_nY+h_n^2(X^2+Y^2)$, direct expansion gives
\[
C_n=1-h_nc_0Y+ih_nX+O(h_n^2),\qquad
D_n=h_nd_0Y+O(h_n^2).
\]
Hence the center and radius of $\Omega_n(X,Y)$ converge to $(X,c_0Y)$ and $d_0Y$, respectively.
Denote the limiting disk by $\Omega(X,Y)$; its minimum height is
$(c_0-d_0)Y=(1-\rho_0)Y/(1+\rho_0)>0$.
These disks eventually lie in a common compact set $K\Subset\mathbb H$, and their indicator functions converge pointwise away from the limiting circle.
Therefore,
\begin{equation}\label{eq:single-domains}
\mathbf1_{\Omega_n(X,Y)}\longrightarrow\mathbf1_{\Omega(X,Y)}
\quad\text{in }L^1(\mathbb H).
\end{equation}

Write $|\cdot|_{\rm euc}$ for ordinary area in the upper half-plane.
The Jacobian relative to normalized area under the affine change of variables is $h_n^2/\pi$, so
\[
M_{\rho_0}\varphi(b_n)
=\frac1{|\Omega_n(X,Y)|_{\rm euc}}
\int_{\Omega_n(X,Y)}F_n(x,t)\dd x\dd t.
\]
The left-hand side tends to zero, while the denominator tends to $\pi d_0^2Y^2>0$. Thus the integral on the right tends to zero.
Weak-star convergence and \eqref{eq:single-domains} give
\[
\left|\int_{\Omega_n(X,Y)}F_n-\int_{\Omega(X,Y)}F\right|
\le \norm\varphi_\infty\norm{\mathbf1_{\Omega_n}-\mathbf1_{\Omega}}_1
 +\left|\int_{\Omega(X,Y)}(F_n-F)\right|\longrightarrow0.
\]
Consequently, $\int_{B((X,c_0Y),d_0Y)}F=0$ for all $X,Y$.
Taking $t=c_0Y$ and $q=d_0/c_0=2\rho_0/(1+\rho_0^2)<1$,
Lemma~\ref{lem:halfplane-unique} yields $F=0$.
\end{proof}
We can now complete the proof of Theorem 1.1.
\begin{proof}[Theorem 1.1: \textnormal{(v)}$\Rightarrow$\textnormal{(ii)}]
By \textnormal{(v)}, there is a fixed $\rho_0=\tanh r_*>0$ such that
$M_{\rho_0}\varphi(a)\to0$. Fix any $\rho\in(0,1)$.
If $M_\rho\varphi$ does not tend to zero at the boundary, there exist $\epsilon>0$ and a sequence of centers
$a_n=(1-h_n)\zeta_n$, with $h_n\downarrow0$, such that
$|M_\rho\varphi(a_n)|\ge\epsilon$.
A subsequence with strictly decreasing $h_n$ can be extracted from any sequence approaching the boundary.
Apply the scaling in Lemma~\ref{lem:single-blowup} and pass to a weak-star convergent subsequence to obtain $F_n\overset{w^*}\to0$.

Let $\Omega_n^\rho$ be the scaled preimage of $E_\rho(a_n)$.
The same center and radius calculation as above, with $X=0$ and $Y=1$, gives
\[
\mathbf1_{\Omega_n^\rho}\longrightarrow\mathbf1_{\Omega^\rho}\quad\text{in }L^1,
\qquad
\Omega^\rho=B\left(\left(0,\frac{1+\rho^2}{1-\rho^2}\right),
\frac{2\rho}{1-\rho^2}\right).
\]
Hence
\[
\left|\int_{\Omega_n^\rho}F_n\right|
\le\left|\int_{\Omega^\rho}F_n\right|
+\norm\varphi_\infty\norm{\mathbf1_{\Omega_n^\rho}-\mathbf1_{\Omega^\rho}}_1\longrightarrow0.
\]
Since $|\Omega_n^\rho|_{\rm euc}\to|\Omega^\rho|_{\rm euc}>0$, the exact change-of-variables identity gives
\[
M_\rho\varphi(a_n)=
\frac{\int_{\Omega_n^\rho}F_n}{|\Omega_n^\rho|_{\rm euc}}\longrightarrow0,
\]
a contradiction. Thus the averages vanish at the boundary for every fixed radius, proving \textnormal{(ii)}.

We have already proved that \textnormal{(i)}, \textnormal{(ii)}, and \textnormal{(iii)} are equivalent.
Clearly, \textnormal{(ii)}$\Rightarrow$\textnormal{(iv)}$\Rightarrow$\textnormal{(v)}.
Together with \textnormal{(v)}$\Rightarrow$\textnormal{(ii)}, just established, this proves the equivalence of all five conditions.
In particular, for any fixed $r_0>0$ prescribed in advance, vanishing of the corresponding local averages suffices for compactness.
\end{proof}
\section{Proof of Theorem 1.2}\label{sec:carleson}
This section proves the Carleson box criterion for general bounded complex-valued symbols.
Necessity follows from automorphism-induced unitary transformations and $L^1$ convergence of the weights associated with Carleson box averages.
For sufficiency, we rescale the symbol near the boundary to the upper half-plane and use a uniqueness lemma for box integrals to exclude nonzero weak-star limits,
then reduce the problem to the Bergman disk condition in Theorem 1.1.
A direct calculation gives the normalized area of a Carleson box:
\begin{equation}\label{eq:boxarea}
|S(\theta,h)|=\frac1\pi\int_{\theta-h}^{\theta+h}\int_{1-h}^1r\dd r\dd t
=\frac{2h^2}{\pi}\left(1-\frac h2\right).
\end{equation}

\subsection{Necessity in Theorem 1.2}
\begin{proof}
Suppose that $T_\varphi$ is compact, and write $M=\norm\varphi_\infty$.
Take arbitrary $h_n\downarrow0$ and an arbitrary sequence of angles $\theta_n$. We will prove that
$C_{h_n}\varphi(\theta_n)\to0$, which gives the required uniform angular limit.
Set
\[
b_n=1-h_n,\qquad
\tau_n(w)=e^{i\theta_n}\sigma_{b_n}(w),\qquad
\Phi_n(w)=\varphi(\tau_n(w)).
\]
Define
\[
(V_nf)(z)=f(\tau_n^{-1}(z))(\tau_n^{-1})'(z).
\]
By the area change-of-variables formula, $V_n$ is unitary on $\A$. For each fixed analytic polynomial $p$,
$\norm{V_np}_2=\norm p_2$. Moreover, since
\[
\tau_n^{-1}(z)=\sigma_{b_n}(e^{-i\theta_n}z),
\qquad
|(\tau_n^{-1})'(z)|=\frac{1-b_n^2}{|1-b_ne^{-i\theta_n}z|^2},
\]
for every fixed $R<1$,
\[
\sup_{|z|\le R}|V_np(z)|
\le\norm p_\infty\frac{1-b_n^2}{(1-R)^2}\longrightarrow0.
\]
As in Lemma~\ref{lem:weak}, density of the reproducing kernels gives $V_np\rightharpoonup0$.
Compactness implies $\norm{T_\varphi V_np}_2\to0$, so for fixed analytic polynomials $p,q$,
\begin{align*}
\int_\D\Phi_n(w)p(w)\overline{q(w)}\dA(w)
&=\langle T_\varphi V_np,V_nq\rangle\longrightarrow0.
\end{align*}
Mixed polynomials are dense in $L^1(\D)$, and $\norm{\Phi_n}_\infty\le M$.
Thus the same approximation argument as in Lemma~\ref{lem:weakstar} gives
\begin{equation}\label{eq:boxnecess-weakstar}
\Phi_n\overset{w^*}{\longrightarrow}0\quad\text{in }L^\infty(\D).
\end{equation}

Write
\[
\Lambda_h=\{w\in\D:\sigma_{1-h}(w)\in S(0,h)\}.
\]
Since rotations preserve area, $\tau_n^{-1}(S(\theta_n,h_n))=\Lambda_{h_n}$.
By \eqref{eq:boxarea} and the derivative formula for disk automorphisms,
\[
|S(0,h)|=\frac{h^2(2-h)}\pi,
\qquad
|\sigma_{1-h}'(w)|^2=
\frac{h^2(2-h)^2}{|1-(1-h)w|^4}.
\]
Therefore, defining
\begin{equation}\label{eq:boxnecess-weight}
H_h(w)=\pi(2-h)\frac{\mathbf1_{\Lambda_h}(w)}{|1-(1-h)w|^4},
\end{equation}
we obtain the identity
\begin{equation}\label{eq:boxnecess-pairing}
C_{h_n}\varphi(\theta_n)=\int_\D\Phi_n(w)H_{h_n}(w)\dA(w).
\end{equation}
In particular, $H_h\ge0$ and $\int_\D H_h\dA=1$.

If $w\in\Lambda_h$, let $b=1-h$ and $z=\sigma_b(w)\in S(0,h)$.
Write $z=re^{it}$, where $1-h<r<1$ and $|t|<h$. Then
\[
|1-z|\le(1-r)+|1-e^{it}|\le2h,
\qquad |1-bz|\le |1-z|+h|z|\le3h.
\]
The automorphism identity gives
\[
1-bw=\frac{1-b^2}{1-bz},
\qquad |1-bw|\ge\frac{h(2-h)}{3h}\ge\frac13.
\]
This yields the uniform bound on the entire disk
\begin{equation}\label{eq:boxnecess-domination}
0\le H_h(w)\le162\pi\qquad(w\in\D,\ 0<h<1).
\end{equation}

For fixed $w\in\D$, write $z_h=\sigma_{1-h}(w)$. Then
\begin{equation}\label{eq:boxnecess-cayley}
\frac{1-z_h}{h}=\frac{1+w}{1-(1-h)w}
\longrightarrow Q(w):=\frac{1+w}{1-w}.
\end{equation}
and $\Re Q(w)=(1-|w|^2)/|1-w|^2>0$.
Since $z_h=1-hQ(w)+o(h)$, choosing a continuous branch of the argument near $1$ gives
\[
\frac{1-|z_h|}{h}\longrightarrow\Re Q(w),
\qquad
\frac{\arg z_h}{h}\longrightarrow-\Im Q(w).
\]
Thus, except where $\Re Q(w)=1$ or $\Im Q(w)=\pm1$,
\[
\mathbf1_{\Lambda_h}(w)\longrightarrow\mathbf1_\Lambda(w),
\quad
\Lambda=\{w\in\D:0<\Re Q(w)<1,\ |\Im Q(w)|<1\}.
\]
These exceptional sets have area zero: $Q$ is a smooth conformal diffeomorphism from the disk onto the right half-plane,
and the preimages of the relevant lines are circular arcs or line segments, each of area zero in local coordinate charts.
Consequently, almost everywhere,
\[
H_h(w)\longrightarrow H(w):=
\frac{2\pi\mathbf1_\Lambda(w)}{|1-w|^4}.
\]
By \eqref{eq:boxnecess-domination} and $|\D|=1$, the dominated convergence theorem gives
\begin{equation}\label{eq:boxnecess-L1}
H\in L^1(\D),\qquad \norm{H_h-H}_1\longrightarrow0.
\end{equation}
Note that $Q'(w)=2/(1-w)^2$. Thus,
in the $Q$ coordinates, $H(w)\mathrm dA(w)$ is exactly one-half of ordinary area measure on the rectangle
$\{0<\Re Q<1,\ |\Im Q|<1\}$.
Its total mass is $1$, consistent with $\int H_h\dA=1$.

By \eqref{eq:boxnecess-pairing},
\[
|C_{h_n}\varphi(\theta_n)|
\le\abs{\int_\D\Phi_nH\dA}+M\norm{H_{h_n}-H}_1\longrightarrow0.
\]
The first term vanishes by \eqref{eq:boxnecess-weakstar}, and the second by \eqref{eq:boxnecess-L1}.
If \eqref{eq:boxcondition} failed, we could choose $\varepsilon>0$, a decreasing sequence $h_n\to0$,
and angles $\theta_n$ such that $|C_{h_n}\varphi(\theta_n)|\ge\varepsilon$,
contradicting the conclusion above. Thus the Carleson box condition holds, proving necessity.
\end{proof}
\subsection{Sufficiency in Theorem 1.2}
In what follows, write $\mathbb H=\mathbb R\times(0,\infty)$, with product Lebesgue measure
$\dd x\dd t$.

\begin{lemma}\label{lem:boxunique}
Let $F\in L^\infty(\mathbb H)$. If, for every $x\in\mathbb R$ and $h>0$,
\begin{equation}\label{eq:zerobox}
\int_{x-h}^{x+h}\int_0^hF(u,t)\dd t\dd u=0,
\end{equation}
then $F=0$ almost everywhere.
\end{lemma}
\begin{proof}
Write $M=\norm F_\infty$. By Fubini's theorem, there is a full-measure set of horizontal coordinates $x$
such that, for each of these $x$, the function $t\mapsto F(x,t)$ is integrable on bounded intervals and has absolute value at most $M$ almost everywhere.
On this set, define
\[
G_h(x)=\int_0^hF(x,t)\dd t,
\]
On the exceptional null set, set $G_h(x)=0$. Then, for all $h,k>0$,
\begin{equation}\label{eq:heightlip}
\norm{G_h}_\infty\le Mh,
\qquad \norm{G_h-G_k}_\infty\le M|h-k|.
\end{equation}

By assumption,
\[
\int_{x-h}^{x+h}G_h(u)\dd u=0\qquad(x\in\mathbb R).
\]
The left-hand side is locally absolutely continuous in $x$, with almost-everywhere derivative
$G_h(x+h)-G_h(x-h)$. Therefore
$G_h(x+2h)=G_h(x)$ almost everywhere, so $G_h$ has a $2h$-periodic representative.
Taking $x=h$ also gives $\int_0^{2h}G_h(u)\dd u=0$.

Fix $h>0$. For $m\in\mathbb Z\setminus\{0\}$, set
\[
\lambda=\frac{m\pi}{h},\qquad
c_m(h)=\frac1{2h}\int_0^{2h}G_h(x)e^{-i\lambda x}\dd x.
\]
Since the integrand is $2h$-periodic,
\begin{equation}\label{eq:longmean}
c_m(h)=\lim_{T\to\infty}\frac1{2T}\int_{-T}^T G_h(x)e^{-i\lambda x}\dd x.
\end{equation}
If $k/h$ is irrational, then $e^{-2ik\lambda}=e^{-2\pi i mk/h}\ne1$.
Using the $2k$-periodicity of $G_k$, for every positive integer $N$ we have
\[
\int_0^{2kN}G_k(x)e^{-i\lambda x}\dd x
=\left(\int_0^{2k}G_k(x)e^{-i\lambda x}\dd x\right)
 \sum_{\nu=0}^{N-1}e^{-2ik\lambda\nu}.
\]
For fixed $h,m,k$, the geometric partial sums satisfy
\[
\abs{\sum_{\nu=0}^{N-1}e^{-2ik\lambda\nu}}
\le\frac{2}{|1-e^{-2ik\lambda}|},
\]
Thus the right-hand side is bounded independently of $N$. 
The negative half-axis is treated in the same way. On the remaining intervals of length less than one period,
the absolute value of the integral is also bounded by a constant independent of $T$. Hence
\[
\lim_{T\to\infty}\frac1{2T}\int_{-T}^T G_k(x)e^{-i\lambda x}\dd x=0.
\]
Combining \eqref{eq:heightlip} and \eqref{eq:longmean}, we obtain
\[
|c_m(h)|\le\norm{G_h-G_k}_\infty\le M|h-k|.
\]
Only now do we let $k\to h$, keeping $k/h$ irrational, to conclude that $c_m(h)=0$.
The zeroth coefficient already vanishes because the mean over a period is zero.

Since $G_h\in L^2(0,2h)$, vanishing of all its Fourier coefficients implies $G_h=0$ almost everywhere.
For completeness, uniqueness follows as follows: on the circle, finite linear combinations of $e^{im\pi x/h}$ are closed under complex conjugation,
contain constants, and separate points. The Stone--Weierstrass theorem  therefore makes them uniformly dense in the continuous periodic functions.
Continuous periodic functions are dense in $L^2(0,2h)$, so the span of these exponentials is dense in $L^2$.
Consequently, $G_h$, being orthogonal to all of them, must vanish.

Removing the union of the corresponding null sets for all positive rational $h$, we find that, for almost every $x$,
$G_h(x)=0$ for every positive rational $h$. By the pointwise integral version of \eqref{eq:heightlip},
the map $h\mapsto G_h(x)$ is continuous, so the identity holds for every real $h>0$.
This map is absolutely continuous on finite height intervals, with almost-everywhere derivative $F(x,h)$.
Thus $F(x,h)=0$ almost everywhere, and another application of Fubini's theorem gives the conclusion.
\end{proof}

\begin{lemma}\label{lem:boxblowup}
Suppose that $\varphi$ satisfies \eqref{eq:boxcondition}. For arbitrary $h_n\downarrow0$ and
$\theta_n\in\mathbb R$, define on $\mathbb H$
\begin{equation}\label{eq:scaledphi}
F_n(x,t)=
\begin{cases}
\varphi\bigl((1-h_nt)e^{i(\theta_n+h_nx)}\bigr),&0<t<1/h_n,\\
0,&t\ge1/h_n.
\end{cases}
\end{equation}
Then every weak-star convergent subsequence has zero limit in $L^\infty(\mathbb H)$;
moreover, every subsequence has a further weak-star convergent subsequence.
\end{lemma}
\begin{proof}
Write $M=\norm\varphi_\infty$. On each coordinate chart without angular wrapping, the scaling map is a nondegenerate smooth transformation.
These charts can be chosen to form a countable cover. Hence the pullback symbol is well defined almost everywhere, and
$\norm{F_n}_\infty\le M$.
By the Banach--Alaoglu theorem and the separability of $L^1(\mathbb H)$,
bounded balls in $L^\infty(\mathbb H)$ are compact and metrizable in the weak-star topology.
We can therefore extract a weak-star convergent subsequence.
Suppose that along such a subsequence, still indexed by $n$,
\[
F_n\overset{w^*}{\longrightarrow}F\quad\text{in }L^\infty(\mathbb H).
\]

Fix arbitrary $x\in\mathbb R$ and $H>0$. For sufficiently large $n$, we have $Hh_n<1$,
and the scaled coordinates map the rectangle $(x-H,x+H)\times(0,H)$ bijectively onto
$S(\theta_n+h_nx,Hh_n)$, with area Jacobian
\[
\mathrm dA=\frac{h_n^2}{\pi}(1-h_nt)\dd u\dd t.
\]
By \eqref{eq:boxarea},
\begin{align}
&\int_{x-H}^{x+H}\int_0^H F_n(u,t)(1-h_nt)\dd t\dd u\notag\\
&\quad=\frac\pi{h_n^2}\int_{S(\theta_n+h_nx,Hh_n)}\varphi\dA\notag\\
&\quad=2H^2\left(1-\frac{Hh_n}{2}\right)
 C_{Hh_n}\varphi(\theta_n+h_nx)\longrightarrow0.
\label{eq:scaledbox}
\end{align}
The last step uses precisely the angular uniformity in condition \eqref{eq:boxcondition}.
On the other hand, omitting the factor $1-h_nt$ from the Jacobian introduces an error of at most
\[
Mh_n\int_{x-H}^{x+H}\int_0^Ht\dd t\dd u=Mh_nH^3\longrightarrow0.
\]
Therefore,
\[
\int_{x-H}^{x+H}\int_0^HF_n(u,t)\dd t\dd u\longrightarrow0.
\]
Although the rectangle meets the boundary $t=0$, it has finite area, so its indicator belongs to $L^1(\mathbb H)$.
The subsequence converges weak-star against all of $L^1(\mathbb H)$.
Thus weak-star convergence gives
\[
\int_{x-H}^{x+H}\int_0^HF(u,t)\dd t\dd u=0.
\]
The same weak-star convergent subsequence works for all $L^1$ test functions, so the conclusion holds for all $x,H$,
without selecting new subsequences for different rectangles. Lemma~\ref{lem:boxunique} now applies, giving $F=0$ almost everywhere.
\end{proof}
\begin{proof}[Proof of sufficiency in Theorem 1.2]
Fix $0<\rho<1$. Suppose, to the contrary, that \eqref{eq:boximpliesdisk} fails. Then there exist
$\varepsilon_0>0$ and a sequence with $|a_n|\to1$ such that
\begin{equation}\label{eq:boxcontradiction}
\abs{\frac1{|E_\rho(a_n)|}\int_{E_\rho(a_n)}\varphi\dA}\ge\varepsilon_0
\qquad\text{for all }n.
\end{equation}
Discarding finitely many terms, assume $a_n\ne0$. Write
\[
a_n=b_ne^{i\theta_n},\qquad b_n=1-h_n,\qquad h_n=1-|a_n|\to0.
\]
If necessary, pass to a further subsequence such that $h_n\downarrow0$. Define $F_n$ by \eqref{eq:scaledphi}.
By Lemma~\ref{lem:boxblowup}, after taking another subsequence,
\begin{equation}\label{eq:scaledzero}
F_n\overset{w^*}{\longrightarrow}0\quad\text{in }L^\infty(\mathbb H).
\end{equation}

Rotating $E_\rho(a_n)$ through the angle $-\theta_n$ gives a Euclidean disk with center and radius
\[
c_n=\frac{(1-\rho^2)b_n}{1-\rho^2b_n^2},
\qquad d_n=\frac{\rho(1-b_n^2)}{1-\rho^2b_n^2}.
\]
In the coordinate strip $|x|<\pi/h_n$, $0<t<1/h_n$, which avoids angular wrapping, define
\[
\Omega_n=\left\{(x,t):
|x|<\frac\pi{h_n},\quad 0<t<\frac1{h_n},\quad
\left|(1-h_nt)e^{ih_nx}-c_n\right|<d_n\right\}.
\]
For sufficiently large $n$, the original Euclidean disk lies near the positive real axis and does not meet the negative-real-axis cut of these coordinates.
Hence these coordinates give a bijection with $E_\rho(a_n)$.

A direct calculation gives
\begin{align*}
\frac{1-c_n}{h_n}
&=\frac{1+\rho^2b_n}{1-\rho^2b_n^2}
\longrightarrow c_\rho:=\frac{1+\rho^2}{1-\rho^2},\\
\frac{d_n}{h_n}
&=\frac{\rho(1+b_n)}{1-\rho^2b_n^2}
\longrightarrow d_\rho:=\frac{2\rho}{1-\rho^2}.
\end{align*}
In particular, $c_\rho-d_\rho=(1-\rho)/(1+\rho)>0$.
Set
\[
\Omega_\rho=\{(x,t):x^2+(t-c_\rho)^2<d_\rho^2\}.
\]
This is a fixed disk whose closure lies in $\mathbb H$.

We first verify that all the $\Omega_n$, after discarding finitely many terms, lie in a common compact set.
The radial coordinate of each of their points satisfies
\[
\frac{1-c_n-d_n}{h_n}<t<\frac{1-c_n+d_n}{h_n}.
\]
The two endpoints converge to $c_\rho-d_\rho>0$ and $c_\rho+d_\rho<\infty$, respectively.
Since $c_n>d_n>0$ for large $n$, the absolute value of the polar angle of a point in the disk with center $c_n$ and radius $d_n$
is at most $\arcsin(d_n/c_n)$. Consequently,
\[
|x|\le\frac1{h_n}\arcsin(d_n/c_n),
\]
The right-hand side is uniformly bounded. Thus there is a fixed compact rectangle $K\Subset\mathbb H$ containing $\Omega_\rho$ and all sufficiently large $\Omega_n$.

Taylor expansion uniformly on $K$ gives
\[
\frac{(1-h_nt)e^{ih_nx}-c_n}{h_n}
=\frac{1-c_n}{h_n}-t+ix+O_K(h_n)
\longrightarrow c_\rho-t+ix.
\]
Together with $d_n/h_n\to d_\rho$, this implies that, at every point outside the circle $\partial\Omega_\rho$,
$\mathbf1_{\Omega_n}\to\mathbf1_{\Omega_\rho}$.
The limiting circle has two-dimensional measure zero, and the common support $K$ has finite area. The dominated convergence theorem therefore gives
\begin{equation}\label{eq:regionL1}
\norm{\mathbf1_{\Omega_n}-\mathbf1_{\Omega_\rho}}_{L^1(\mathbb H)}\longrightarrow0.
\end{equation}

Set
\[
J_n=\int_{\Omega_n}(1-h_nt)\dd x\dd t,
\qquad
W_n(x,t)=\frac{(1-h_nt)\mathbf1_{\Omega_n}(x,t)}{J_n}.
\]
Here $1-h_nt>0$ on $\Omega_n$. By the common compact support and \eqref{eq:regionL1},
\[
J_n\longrightarrow J:=\int_{\Omega_\rho}\dd x\dd t=\pi d_\rho^2>0,
\]
More explicitly, write $A_n=(1-h_nt)\mathbf1_{\Omega_n}$, $A=\mathbf1_{\Omega_\rho}$, and $W=A/J$,
and let $T_K=\sup\{t:(x,t)\in K\}$ and $|K|_{\rm E}=\int_K\dd x\dd t$.
Then
\[
\norm{A_n-A}_1
\le\norm{\mathbf1_{\Omega_n}-\mathbf1_{\Omega_\rho}}_1
+h_nT_K|K|_{\rm E}\longrightarrow0.
\]
Hence $|J_n-J|\le\norm{A_n-A}_1$, and $J_n\ge J/2$ for sufficiently large $n$. Therefore,
\[
\norm{W_n-W}_1
\le\frac{\norm{A_n-A}_1+|J_n-J|}{J_n}
\le\frac4J\norm{A_n-A}_1\longrightarrow0,
\]
that is,
\begin{equation}\label{eq:weightboxL1}
W_n\longrightarrow W=\frac{\mathbf1_{\Omega_\rho}}{J}
\quad\text{in }L^1(\mathbb H).
\end{equation}
Note that $J$ denotes ordinary Euclidean area in the upper half-plane, as distinct from normalized area on the disk.

The area change of variables gives
\begin{align*}
\frac1{|E_\rho(a_n)|}\int_{E_\rho(a_n)}\varphi\dA
&=\frac{\int_{\Omega_n}F_n(x,t)(1-h_nt)\dd x\dd t}
 {\int_{\Omega_n}(1-h_nt)\dd x\dd t}\\
&=\int_{\mathbb H}F_n(x,t)W_n(x,t)\dd x\dd t.
\end{align*}
By \eqref{eq:scaledzero} and \eqref{eq:weightboxL1},
\[
\abs{\int_{\mathbb H}F_nW_n}
\le\abs{\int_{\mathbb H}F_nW}
+\norm\varphi_\infty\norm{W_n-W}_1\longrightarrow0.
\]
This contradicts \eqref{eq:boxcontradiction}, so \eqref{eq:boximpliesdisk} holds.

Finally, $\rho\in(0,1)$ was arbitrary and fixed, so condition
\textnormal{(ii)} of Theorem 1.1 holds. The sufficiency part of that theorem implies that $T_\varphi$ is compact.
\end{proof}

\section{A counterexample: angular pointwise vanishing without compactness}
An interesting problem is whether angular pointwise vanishing of the local averages of the symbol over Carleson boxes is equivalent to compactness of the Toeplitz operator. The following counterexample shows that angular pointwise vanishing of these local averages does not imply compactness.

\begin{theorem}[Angular pointwise vanishing without compactness]\label{thm:box-pointwise-counterexample}
Let
\[
\theta_n=2^{-n},\qquad h_n=2^{-n^2-10},\qquad n\ge1,
\]
and define
\[
E_n=\{re^{it}:1-h_n<r<1-h_n/2,\quad
                  |t-\theta_n|<h_n/4\},\qquad
\varphi=\mathbf1_{\bigcup_{n\ge1}E_n}.
\]
Then $\varphi$ is a measurable bounded symbol taking values in $\{0,1\}$, with the following properties:
\begin{enumerate}[label=\textnormal{(\roman*)}]
\item For every fixed $\theta\in\mathbb R$,
\begin{equation}\label{eq:counter-pointwise}
\lim_{h\downarrow0}\frac1{|S(\theta,h)|}
                  \int_{S(\theta,h)}\varphi\dA=0.
\end{equation}
\item This convergence is not uniform in $\theta$. More precisely, for every $n$,
\begin{equation}\label{eq:counter-nonuniform}
\sup_\theta\frac1{|S(\theta,h_n)|}
                  \int_{S(\theta,h_n)}\varphi\dA\ge\frac1{32}.
\end{equation}
\item $T_\varphi$ is not compact on $\A$.
\end{enumerate}
\end{theorem}

\begin{proof}
Each $E_n$ is measurable, so the indicator of their countable union is measurable and
$0\le\varphi\le1$. Since $h_n\le\theta_n$,
all angular intervals $(\theta_n-h_n/4,\theta_n+h_n/4)$ lie in $(0,1)$,
so there is no ambiguity concerning the branch of the angle. Integration in polar coordinates with normalized area gives
\begin{align}
|E_n|
&=\frac1\pi\int_{\theta_n-h_n/4}^{\theta_n+h_n/4}
                \int_{1-h_n}^{1-h_n/2}r\dd r\dd t\notag\\
&=\frac{h_n^2}{4\pi}\left(1-\frac{3h_n}{4}\right).
\label{eq:counter-area}
\end{align}
In particular, since $0<h_n<1$,
\begin{equation}\label{eq:counter-area-bounds}
\frac{h_n^2}{16\pi}\le |E_n|\le\frac{h_n^2}{4\pi}.
\end{equation}
For every $0<h<1$, the box area satisfies
\begin{equation}\label{eq:counter-box-bounds}
|S(\theta,h)|=\frac{h^2(2-h)}\pi,\qquad
\frac{h^2}\pi\le |S(\theta,h)|\le\frac{2h^2}\pi.
\end{equation}

\medskip
Let $d_{\mathbb T}$ denote the shortest angular distance on the circle. Fix
$\theta\not\equiv0\pmod{2\pi}$, and write $d=d_{\mathbb T}(\theta,0)>0$.
Since $\theta_n+h_n/4\to0$, choose $N_0\ge2$ so that, for $n\ge N_0$,
the angle $t$ of every point in $E_n$ satisfies $d_{\mathbb T}(t,0)<d/4$.
The triangle inequality gives $d_{\mathbb T}(t,\theta)>3d/4$,
so $S(\theta,h)$ is disjoint from these tail sets whenever $h<d/4$.
On the other hand, if
\[
h<\min_{1\le n<N_0}h_n/2,
\]
then every point in the box has modulus greater than $1-h>1-h_n/2$, so the box is also disjoint from the finitely many sets
$E_1,\ldots,E_{N_0-1}$.
Thus, in this fixed direction, the box average is identically zero at all sufficiently small scales.

\medskip
Consider $0<h<1/2$. If $E_n\cap S(0,h)\ne\varnothing$,
the chosen angular branch implies
\[
\theta_n-h_n/4<h.
\]
Since $h_n\le\theta_n$, it follows that $\theta_n<2h$.
Choose the unique integer $N\ge1$ such that
\[
2^{-(N+1)}<h\le2^{-N}.
\]
The inequalities $2^{-n}<2h\le2^{1-N}$ imply $n\ge N$.
Thus subadditivity and \eqref{eq:counter-area-bounds}--\eqref{eq:counter-box-bounds}
give
\begin{equation}\label{eq:counter-tail-start}
0\le\frac1{|S(0,h)|}\int_{S(0,h)}\varphi\dA
\le\frac1{4h^2}\sum_{n\ge N}h_n^2.
\end{equation}
Note that
\[
\frac{h_{n+1}^2}{h_n^2}=2^{-4n-2}\le\frac1{64}\quad(n\ge1),
\]
Hence
\[
\sum_{n\ge N}h_n^2\le\frac{64}{63}\,2^{-2N^2-20}.
\]
Substituting into \eqref{eq:counter-tail-start} and using $h^{-2}<2^{2N+2}$, we obtain
\[
0\le\frac1{|S(0,h)|}\int_{S(0,h)}\varphi\dA
\le\frac{64}{63}\,2^{-2N^2+2N-20}\longrightarrow0.
\]
As $h\downarrow0$, we have $N\to\infty$. This proves the limit along \emph{all} scales tending to zero. By angular periodicity,
\eqref{eq:counter-pointwise} holds for every fixed angle.

\medskip

By definition, $E_n\subset S(\theta_n,h_n)$. Therefore,
\begin{align*}
\frac1{|S(\theta_n,h_n)|}\int_{S(\theta_n,h_n)}\varphi\dA
&\ge\frac{|E_n|}{|S(\theta_n,h_n)|}\\
&=\frac{1-3h_n/4}{4(2-h_n)}\ge\frac1{32}.
\end{align*}
The last inequality also follows directly from
\eqref{eq:counter-area-bounds}--\eqref{eq:counter-box-bounds}.
This area ratio tends to $1/8$.
Since $h_n\to0$, angular uniform convergence fails.

\medskip
Take
\[
a_n=\left(1-\frac{3h_n}{4}\right)e^{i\theta_n},\qquad
k_a(z)=\frac{1-|a|^2}{(1-\overline a z)^2}.
\]
Write $b_n=1-3h_n/4$. For $z=re^{it}\in E_n$,
set $u=t-\theta_n$. Then $|u|<h_n/4$ and $1-r<h_n$, so
\begin{align*}
|1-\overline{a_n}z|
&=|1-b_nr e^{iu}|\\
&\le (1-b_nr)+b_nr|1-e^{iu}|\\
&\le (1-b_n)+(1-r)+|u|\le2h_n.
\end{align*}
Moreover,
\[
1-|a_n|^2=(1-b_n)(1+b_n)\ge\frac{3h_n}{4}.
\]
Thus, for $z\in E_n$,
\[
|k_{a_n}(z)|^2
=\frac{(1-|a_n|^2)^2}{|1-\overline{a_n}z|^4}
\ge\frac{9}{256h_n^2}.
\]
Using $\varphi\ge0$, the identity $\varphi=1$ on $E_n$, and
\eqref{eq:counter-area-bounds}, we obtain
\begin{equation}\label{eq:counter-kernel-lower}
\langle T_\varphi k_{a_n},k_{a_n}\rangle
=\int_\D\varphi|k_{a_n}|^2\dA
\ge\int_{E_n}|k_{a_n}|^2\dA
\ge\frac9{4096\pi}>0.
\end{equation}

On the other hand, $k_{a_n}\rightharpoonup0$ as $|a_n|\to1$.
Indeed, for any analytic polynomial $p$, the reproducing property gives
\[
|\langle p,k_{a_n}\rangle|
=(1-|a_n|^2)|p(a_n)|\longrightarrow0.
\]
Analytic polynomials are dense in $\A$, and $\norm{k_{a_n}}=1$.
For any $f\in\A$, polynomial approximation in norm followed by the Cauchy--Schwarz inequality
therefore yields $\langle f,k_{a_n}\rangle\to0$.
If $T_\varphi$ were compact, it would map this bounded weakly null sequence to a norm-null sequence. Hence
\[
|\langle T_\varphi k_{a_n},k_{a_n}\rangle|
\le\norm{T_\varphi k_{a_n}}\longrightarrow0,
\]
contradicting \eqref{eq:counter-kernel-lower}.
Therefore $T_\varphi$ is not compact.
\end{proof}
\subsection*{Data availability statement}
No datasets were generated or analyzed during the current study.

\subsection*{Conflicts of interest}
The authors declare that they have no competing interests.

\medskip
\textbf{Acknowledgments.}
G.~F. Cao was supported by the National Natural Science Foundation of China (Grant No.~12071155).
L. He was supported by the National Natural Science Foundation of China (Grant No.~12371127).


Cao: \texttt{guangfucao@163.com}

He: \texttt{helichangsha1986@163.com}

Zhang: \texttt{zsq0225@163.com}
\end{document}